\documentclass{article}
\usepackage[a4paper]{geometry}

\usepackage{algorithmicx}
\usepackage{algorithm}

\usepackage{algpseudocode}
\usepackage{graphicx}
\usepackage{tikz}
\usepackage{color,hyperref}
\definecolor{darkblue}{rgb}{0.0,0.0,0.5}
\hypersetup{colorlinks,breaklinks,linkcolor=darkblue,urlcolor=darkblue,anchorcolor=darkblue,citecolor=darkblue}

\usepackage{makecell}

\usepackage[utf8]{inputenc} 

\usepackage{graphics} 
\usepackage{epsfig} 
\usepackage{cite}
\usepackage{amsmath} 
\usepackage{amssymb}  
\usepackage{amsthm}
\usepackage{cite}
\usepackage{pgf-pie}
\usepackage{pgfplots}
\usepackage{tikz}
\usetikzlibrary{matrix,positioning}
\tikzset{bullet/.style={circle,fill,inner sep=2pt}}

\newtheorem{theorem}{Theorem}[section]
\newtheorem{lemma}{Lemma}[section]

\newtheorem{remark}{Remark}[section]

\newtheorem{example}{Example}[section]
\newtheorem{problem}{Problem}[section]
\usepackage{verbatim}
\usepackage{graphicx} 
\usepackage{epstopdf}
\usepackage{color}
\usepackage{multirow}
\usepackage{lipsum}
\makeatletter
\def\footnoterule{\relax%
	\kern-5pt
	\hbox to \columnwidth{\hfill\vrule width .9\columnwidth height 0.4pt\hfill}
	\kern4.6pt}
\makeatother

\usetikzlibrary{decorations.markings}

\tikzset{
	tangent/.style={ 
		decoration={
			markings,
			mark=
			at position #1
			with
			{
				\coordinate (tangent point-\pgfkeysvalueof{/pgf/decoration/mark info/sequence number}) at (0pt,0pt);
				\coordinate (tangent unit vector-\pgfkeysvalueof{/pgf/decoration/mark info/sequence number}) at (1,0pt);
				\coordinate (tangent orthogonal unit vector-\pgfkeysvalueof{/pgf/decoration/mark info/sequence number}) at (0pt,1);
				\draw[blue] (0pt,0pt) coordinate (Tang) -- (20pt,0pt);
			}
		},
		postaction=decorate
	},
	use tangent/.style={
		shift=(tangent point-#1),
		x=(tangent unit vector-#1),
		y=(tangent orthogonal unit vector-#1)
	},
	use tangent/.default=1
}
\makeatletter

\DeclareRobustCommand{\rvdots}{%
	\vbox{
		\baselineskip4\p@\lineskiplimit\z@
		\kern-\p@
		\hbox{.}\hbox{.}\hbox{.}
}}

\begin{document}





\title{On Solving Robust Log-Optimal  Portfolio:\\ A Supporting Hyperplane Approximation Approach}
\author{Chung-Han Hsieh\\ Department of Quantitative Finance, \\National Tsing Hua University, Hsinchu 300044, Taiwan R.O.C.\\
	\href{mailto: ch.hsieh@mx.nthu.edu.tw}{ch.hsieh@mx.nthu.edu.tw}} 
\date{\vspace{-5ex}}
\maketitle



\begin{abstract}
A \textit{log-optimal} portfolio is any portfolio that maximizes the expected logarithmic growth (ELG) of an investor's wealth.
This maximization problem typically assumes that the information of the true distribution of returns is known to the trader in advance. 
However, in practice, the return distributions are indeed \textit{ambiguous}; i.e., the true distribution is unknown to the trader or it is partially known at best.
To this end, a \textit{distributional robust log-optimal portfolio problem}  formulation arises naturally. 
While the problem formulation takes into account the ambiguity on return distributions, the problem needs not to be tractable in general.
To address this, in this paper, we propose a  \textit{supporting hyperplane approximation} approach that allows us to reformulate a class of distributional robust log-optimal portfolio problems into a  linear program, which can be solved very efficiently. 
Our framework is flexible enough to allow \textit{transaction costs}, \textit{leverage and shorting}, \textit{survival trades}, and \textit{diversification considerations}.
In addition, given an acceptable approximation error, an efficient algorithm for rapidly calculating the optimal number of hyperplanes is provided.
Some empirical studies using historical stock price data are also provided to support our theory. 
\end{abstract}


{\small \textbf{keywords:} Financial Engineering, Stochastic Systems, Distributionally Robust Optimization, Portfolio Optimization,  Kelly Criterion, Robust Linear Programming, Approximation Theory. } 

\maketitle

%


\section{Introduction}\label{SECTION: Introduction}
In portfolio management, one of the key questions that most investors want to address is how to find an ``optimal" asset allocation fraction so that the desired \textit{risk-reward} objective can be achieved. 
To address this, \cite{Markowitz_1952} and \cite{markowitz1959portfolio} propose the celebrated mean-variance model in a single-period setting.
Since then, many extensions and ramifications are developed along the line of {portfolio theory} and {optimization}; e.g., see \cite{markowitz1957elimination, sharpe1963simplified,  kroll1984mean, roll1992mean,  demiguel2009generalized}.
A good survey on this topic can be found in \cite{steinbach2001markowitz}.
However, the Markowitz-style approach is \textit{static} in the sense that it only optimizes for the next rebalancing.

In contrast to the class of single-period portfolio optimization problems,~\cite{Kelly_1956} proposes 
 an alternative approach called \textit{Kelly criterion}  aimed at addressing \textit{multi-period} betting problem in a repeated gambling setting.
 The theory calls for a maximization of the expected logarithmic growth~(ELG) of a gambler's account; see also \cite{luenberger2013investment} for a good introduction to the Kelly-based approach.
  This framework is readily generalized to stock trading and portfolio optimization scenario; e.g., see \cite{latane1959criteria, rotando1992kelly, li1993growth, thorp1975portfolio, thorp2006kelly, lo2018growth}.
 It is well-known that the Kelly-based approach  guarantees the so-called  {\it comparative optimality} and {\it myopic property}; see~\cite{cover2006elements, maclean2010long, maclean2010good}. That is,  the growth rate of trader's wealth  is maximized asymptotically and the trader who adopted the Kelly-based trading strategy does not have to consider prior nor subsequent investment opportunities. 
 Additionally, the corresponding Kelly-based portfolio also minimizes the expected time to reach a prespecified target account value; e.g., see \cite{breiman1961optimal} and \cite{algoet1988asymptotic}.

 In addition, we mention a sampling of the developments along this line of research.  See \cite{cover1984algorithm} for an algorithm of solving the classical log-optimal portfolio problem,  \cite{bell1988game}  for a discussions on game-theoretic optimal portfolio, \cite{kuhn2010analysis}  for a study on continuous-time log-optimal portfolio.
 A textbook that contains many important papers on the Kelly-based approach can be found in \cite{maclean2011kelly}.
 See also  \cite{wu2018novel} and \cite{wu2022optiontrading} for studies on the application of Kelly-based approach in option trading.
 Recently, \cite{frahm2020statistical} studies the statistical properties of the estimators for the log-optimal portfolio;  \cite{lian2021optimal} studies a optimal growth in a two-sided market, and
 \cite{lototsky2021kelly} for studies on Kelly criterion with continuous L\'evy  process as a model for returns.
 

 Most of the work related to the Kelly-based approach are typically assuming that the return distributions are known in advance. 
 However, in practice, the true distribution of returns is \textit{unknown} or only partially known to the trader at best. 
 Said another way, the return distributions are indeed \textit{ambiguous} to the trader. 
 To this end, some of the work attempt to remedy this \textit{ambiguity} issue. For example,
 \cite{cover1991universal} proposes a \textit{universal portfolio} algorithm that generate an adaptive strategy from historical data.
 The resulting universal portfolio can be shown at least as well as the best log-optimal portfolio selected in hindsight. 
 However, in the short run, the  portfolios might be susceptible to error maximization.
 As another approach, \cite{rujeerapaiboon2016robust} proposes a version of robust log-optimal portfolio framework by maximizing a version of the Value at Risk of portfolio returns under a long-only framework. 
 They show that the problem is indeed a tractable \textit{semidefinite program} (SDP) and with exploiting a certain structure of the ambiguity set, it is possible to obtain a \textit{second-order cone program} (SOCP). 
 Later,  \cite{rujeerapaiboon2018risk} considers a new objective for  the Kelly betting problem using a \textit{conservative} expected value with the similar aim to mitigate the  \textit{ambiguity} issues.   
 Recently, \cite{sun2018distributional} uses a convex optimization approach for solving a class of distributional robust Kelly betting problems with the assumptions that the returns of the gambles are independent and identically distributed~(IID). 
 However, all of the work above assumes that the trades must be long-only and cash-financed.
 In contrast to the existing literature, in this paper, we propose a class of distributional robust log-optimal portfolio problems under a polyhedron ambiguity set for the return distributions.
 In addition, our formulation allows extra flexibility in the sense that  various practical trading requirements such as \textit{transaction costs}, \textit{leveraging and shorting}, \textit{survival trades}, and \textit{diversification considerations} are involved.
 Then, we propose a new hyperplane approximation approach that enables us to solve the distributional robust problem in a much fast \textit{linear program} paradigm. 

\subsection{Contributions of this Paper}
The main contributions of this paper are summarized as follows.
{\noindent
\begin{itemize} 
	\item 
	We consider a class of distributional robust log-optimal portfolio formulations with a polyhedron  ambiguous return distributions. Our formulation is flexible enough to incorporate various practical constraints such as \textit{transaction costs}, \textit{leveraging and shorting}, \textit{survival trades}, and \textit{diversification}.
	\item We provide a supporting hyperplane approximation approach and  prove that such an approximation enables us to reformulate the distributional robust log-optimal portfolio problem as a linear program.  Hence, it can be solved in a very efficient way.	 
	\item	We refine our approximation and study the optimal number of supporting hyperplanes.  An efficient algorithm for rapidly calculating the optimal number of hyperplanes is provided.
	In addition, if there is no ambiguity on the return distributions, we show that our approach attains an approximate optimum which can be arbitrarily close to the true optimum as far as performance is concerned. 
	\item In empirical studies with historical stock price data, we show that our distributional robust log-optimal portfolio is competitive with the classical log-optimal portfolio. 
	\item We also indicate one possibility to extend our formulation. That is, by involving a surrogate \textit{drawdown} risk constraint, we show that the distributional robust log-optimal portfolio problem is still a concave program. Hence, it may be solved in an efficient manner. 
\end{itemize}

\section{Problem Formulation} \label{section: Problem Formulation}
In this section, we first provide some preliminaries and then formulate a distributional robust log-optimal portfolio problem.

\subsection{Finite Outcome Case} Consider a portfolio consisting of~$n \geq 1$ risky assets. 
For  $k=0,1,2,\dots,N-1$ the~$i$th asset at stage~$k$ whose price is denoted by~$S_i(k)>0$.
The associated per-period \textit{rate of returns} for the~$i$th asset is given by
\[
X_i(k) := \frac{S_i(k+1) - S_i(k)}{S_i(k)}
\]
with $X_i(k) >-1.$\footnote{Our setting is flexible enough to involve at least one of the assets, say the $1$st asset, to be \textit{riskless} with nonnegative rate of return $X_1(k):=r_f \geq 0$. That is, if an asset is riskless, the return is assumed to be deterministic and is treated as a degenerate random variable with value $r_f$ for all $k$ with probability one.}
For each $k$, let
$$
X(k) := [X_1(k)\;\; X_2(k)\;\; \cdots\; X_n(k)]^T.
$$
The return vector~$X(k)$ is drawn according to an \textit{unknown} distribution function but are assumed to be identically distributed in $k$ and is supported on only $m$ points.\footnote{While the distribution is unknown, we assume the case where one of $m$ events occurs; i.e., $X(k)$ is supported on only $m$ points. It is worth mentioning that, in the finance literature, the identical distributed returns or even stronger cases such as IID returns are closely related to a market which is \textit{information efficient}; e.g., see~\cite{luenberger2013investment,bodie2018investments, fama2021market}.} 
The corresponding joint probability mass function for the returns is given by 
\[
P(X_1(k) = x_{1}^{j},\dots,X_n(k) = x_{n}^{j}) = p_j, \;\; j\in \{1,2,\dots,m\}
\]  
with $p_j \geq 0$ and $\sum_{j=1}^m p_j =1.$
Or more compactly,
$
P\left( {X\left( k \right) = {x^j}} \right) := {p_j}
$ 
where $x^j := [x_1^j\;x_2^j\; \cdots\;x_n^j]^T$ and $j\in \{1,2,\dots,m\}.$ 
In the sequel, we take  $x_{i,\min} := \min_j x_i^j$ and $  x_{i,\max} := \min_j x_i^j$  for $i=1,2,\dots,n.$
%



\subsection{Linear Trading Policy} 
As far as the computational tractability of a trading policy is concerned,  carryout optimization over all general causal policies is unrealistic.
 Instead, we will restrict attention to a memoryless linear trading policy that keeps the associated portfolio weights constant across all rebalancing stages.
Specifically,
to establish the trading scheme, for $k=0,1,\dots, $ let $V(k)$ be the \textit{account value }at stage~$k$. We consider the so-called \textit{linear} trading policy for the $i$th asset. That is, 
$$
u_i(k) := K_i V(k)
$$
where~$K_i$ is the weight for the $i$th asset.\footnote{The linear policy $u_i(k) = K_i V(k)$ is widely used  in practice.  One can readily convert the linear policy in terms of the corresponding number of shares. Specifically, let $N_i(k)$ be the number of shares invested at stage~$k$. 
	Then, with the price~$S_i(k)>0$, it follows that~$
	N_i(k) := u_i(k)/S_i(k).
	$}
The policy for the portfolio at stage~$k$ is given by 
$$\sum_{i = 1}^n u_i(k) = \sum_{i = 1}^n K_i V(k).$$
In the sequel, we shall take a vector notation
$$
K := [K_1 \; K_2\; \cdots \;K_n]^T.
$$
As seen later in Section~\ref{subsection: constraints considerations},  various practical trading requirements on $K$ that are imposed in our framework are discussed.

\subsection{Account Value Dynamics with Transaction Costs} 
Let $c_i \in (0,1)$ be a \textit{percentage transaction costs} on Asset~$i$.\footnote{
	For example, if one trades in Taiwan Stock Exchange, then a typical transaction cost is $\alpha \cdot 0.1425\%$ of the trade value for some $\alpha \in (0,1)$. As a second example, if one adopts some professional broker services such as Interactive Brokers Pro., it would cost $\$0.005$ per share with minimum fee \$1 dollar and maximum $1\%$ of the trade value. } 
That is, at stage $k$, if one invests $u_i(k)$ at Asset~$i$, then the associated transaction costs in dollar is $|u_i(k)| c_i$. The dynamics of account value at stage $k+1$ is characterized by the following stochastic recursive equation:
\begin{align}
	V(k + 1) 
	&= V(k) + \sum\limits_{i = 1}^n u_i(k) X_i(k) - \sum_{i = 1}^n |u_i(k)| c_i\\
	&= V(k) + \sum_{i = 1}^n u_i(k) \widetilde{X}_i(k)
\end{align}
where 
$$
\widetilde{X}_i(k) := 
\begin{cases}
	X_i(k) - c_i, & K_i \geq 0;\\
	X_i(k) + c_i, & K_i <0
\end{cases}
$$
is the \textit{fee-adjusted returns}.\footnote{Other transaction cost models are possible; e.g., one can add a term that is quadratic in the trade value;  e.g., see \cite{almgren2001optimal,garleanu2013dynamic,  boyd2017multi}.}
Via a straightforward calculation, it follows that the account value at terminal stage~$k=N$ for some integer $N>1$ is given~by 
\[
V\left( N \right) = \prod\limits_{k = 0}^{N-1} {\left( {1 + K^T \widetilde{X}(k)} \right)} V(0).
\]
In the sequel, we may sometimes write $V_K(N, {X})$ instead of~$V(N)$ to emphasize the dependence on feedback gain $K$ and return sequence ${X}:=\{{X}(k): k \geq 0\}.$  If $c_i :=0$ for all $i$, then the account value reduces to $V(N) =  \prod\limits_{k = 0}^{N-1} {\left( {1 + K^T {X}(k)} \right)} V(0)$, which is typically used in the literature; e.g., see \cite{cover2006elements,boyd2017multi, hsieh2018rebalancing, hsieh2020feedback}.

\subsection{Constraints Considerations} \label{subsection: constraints considerations}
In this subsection, we consider three practical constraints to be involved in our formulation. 
That is, $(i)$ shorting and leveraging $(ii)$ survival for all time; i.e.,~$V(k) \geq 0$ for all~$k$ with probability one, and $(iii)$ diversification constraint. 

\subsubsection{Shorting and Leverage Constraints.}
To allow shorting and leverage, for Asset $i$ and $k=0,1,\dots, N-1$, we  write the trading policy $u_i$ into two parts as follows.
\begin{align*}
	u_i(k) 
	&= K_i V(k) \\
	&:= K_{i,L} V(k)+ K_{i,S} V(k)
\end{align*}
where $K_{i,L} \ge 0$ represents the proportion of going \textit{long}  and $K_{i,S} \le 0$ represents the proportion of going \textit{short}. If $K_{i,L} + K_{i,S} \geq 0$, then it indicates one goes long with Asset $i$. 
Similarly, if $K_{i,L} + K_{i,S} \leq 0$, then one goes short with Asset $i$.
We require that $\sum_{i=1}^n |u_i(k)| \leq L V(k)$ for some leverage constant~$L \geq 1$.  
Equivalently, we have
\[
\sum_{i=1}^{n} |K_{i,L} + K_{i,S}| \leq L.
\]
Noting that the constraints above are characterized by various linear inequalities, it forms a convex polytope; see \cite{boyd2004convex}.

\subsubsection{Survival Constraints.}
It is important that the trade is \textit{survival}; i.e., $V(k) <0$ is disallowed for all $k$ with probability one.\footnote{A detailed discussion on the survival condition is referred to~\cite{hsieh2016kelly}. }
This requires that $(1+K^TX(k)) V(0) \geq 0$ for all $k$ with probability one. 
Since $V(0)>0$, it is equivalent to require
\[
 \sum_{i = 1}^n (K_{i,L} + K_{i,S}) X_i(k) \geq -1
\] 
for all $k$ with probability one. Since $x_{i,\min} \leq X_i(k) \leq x_{i,\max}$ for all $i=1,\dots,n$ and~$k=0,1,\dots,N-1$, it follows that
\[
\sum_{i = 1}^n K_{i,L} \min\{x_{i,\min},0\} +\sum_{i=1}^n K_{i,S} \max\{x_{i,\max},0\} \geq -1.
\]
Therefore, the survival constraint  is  imposed as follows:
\begin{align*}
	&\sum_{i = 1}^n {{K_{i,L}} | \min\left\{ { x_{i,\min }  ,0} \right\}}|  - \sum\limits_{i = 1}^n {{K_{i,S}}\max \left\{ {0,{x_{i,\max } }} \right\}}  \le 1.
\end{align*}
Similar to the shorting and leverage constraint, the survival constraints described above form a convex set.

\subsubsection{Diversified Holding Constraints.} 
From a risk management perspective, an exhibit of large concentrations in any specific asset should be avoided;  e.g., see \cite{luenberger2013investment, fabozzi2007robust, bodie2018investments}.
Hence, it is natural to add a constraint restricting the maximal holding weights in each asset.
To this end, we set
\[
K_{i,\min} \leq K_{i,L} + K_{i,S} \leq K_{i,\max}
\]
for some $K_{i,\min}\geq 0$ and $K_{i,\max} \leq 0$ representing the lower and upper bounds of the weights of Asset~$i$.
A typical choice is that $K_{i,\min} =0$ and $K_{i,\max} = L/m$ with $L \geq 1$ being the leverage constant and $m$ being the number of assets in the portfolio.
In the sequel, we shall use $\cal K$ to denote the admissible set of the totality of these three constraints.

\medskip
\begin{remark}\rm
	Since  all three constraints above are formed by various linear inequalities, 
	it is readily verified that the constraint set~$\cal K$ is convex and closed.
\end{remark}

\subsection{Distributional Robust  Log-Optimal Portfolio} \label{subsection: Expected log-portfolio problem}
We are now ready to study a distributional robust version of the log-optimal portfolio problem in which the probability distribution $p$ is unknown. 
Rather, we assume that $p \in \mathcal{P}$, a set of possible return distributions or the so-called \textit{ambiguity set}. 
Said another way, $\cal P$ is the set of all possible return distributions under some available prior information.
For $k=0,1,\dots,N-1$ and $c_i=0$,\footnote{In our setting, the percentage transaction costs can be lumped into original rate of returns to obtain the fee-adjusted returns. In addition, since we assume a linear policy, the log-optimal portfolio optimization problem without transaction costs is equivalent to that with transaction costs in the sense the optimal constant weighting vector $K$ remains unchanged.} consider the \textit{expected logarithmic growth rate}  given by
	\begin{align*}
		g_p(K) 
		&:=  \frac{1}{N} \mathbb{E}_p\left[  \log  \frac{V_K(N, {X})}{ V(0) }\right]\\
		&= \frac{1}{N} \mathbb{E}_p\left[ \sum_{k=0}^{N-1} \log (1+K^T {X}(k))  \right]\\
		&= \frac{1}{N}\sum_{k=0}^{N-1}  \mathbb{E}_p\left[  \log (1+K^T {X}(k))  \right]
	\end{align*}
where the $\mathbb{E}_p[\cdot]$ is the expectation taken with respect to the random vector $X(k)$ given that it follows the joint probabilities $p$. 
The subscript $p$ in $g_p$ is used to emphasize the dependence on the unknown probability~$p$.
	With the assumption that $X(k)$ are identically distributed in~$k$, we have
	\begin{align*}
	g_p(K) 
	&= \mathbb{E}_p[\log (1+K^T {X}(0)) ] \\
	&= \sum_{j = 1}^m p_j \log(1+K^Tx^j).
	\end{align*}
	Our objective is to seek an optimal $K \in \mathcal{K}$ achieving the \textit{distributional robust} expected logarithmic growth rate
	\begin{align}
	g^* := \max_{K \in \mathcal{K}} \; \inf_{p \in \mathcal{P}} g_p(K) 
	\end{align}
		where $\cal{K}$ is the admissible set of $K$ which captures the totality of the constraints described in Sections~\ref{subsection: constraints considerations}.
	Note that $\inf_{p \in \mathcal{P}} g_p(K)$ is concave in $K$ since it is an infimum of a family of concave functions of $K$.
	This is indeed a concave distributional robust optimization problem. 
	However, as indicated in~\cite{nemirovski2009robust}, such problems may not be tractable in general. 
	To this end, in the sequel, our aim is to show how such a problem can be tractably solved for a class of polyhedron probability sets $\mathcal{P}.$ In addition, as seen in later sections to follow, with a hyperplane approximation, the tractability can be improved further.
	
	\medskip
	\begin{remark}\rm
		$(i)$ Return compounds multiplicatively rather than additively. Hence, it is natural to consider maximizing $\mathbb{E}[ \log V(k)/V(0)] $ instead of $\mathbb{E}[V(k)/V(0)]$. 
	$(ii)$	As mentioned previously,~$g(K) = \mathbb{E}[\log(1+K^T {X}(0))]$ is concave in~$K$.
		Hence, if $\cal K$ is convex set, one obtains a concave program. 
	$(iii)$	If $\cal K$ is compact, then, according to Weierstrass Extremum Theorem; see, e.g., \cite{rudin1964principles}, a maximizing $K^*$ exists.
		$(iv)$ It is readily verified that the log-optimal portfolio problem with percentage transaction costs consideration is equivalent to that without the transaction costs in the sense that the optima are the same.
	\end{remark}

	\subsection{Polyhedron Ambiguous Return Distributions} \label{subsection: polyhedron ambigious return distributions}
	Let $p:=[p_1\; p_2\; \cdots p_m]^T \in \mathbb{R}^m$ and take
	$$S_m :=\left\{ p \in \mathbb{R}_+^{m}: \sum_{j = 1}^m p_j =1, p_j \geq 0 \;\; j=1,2,\dots,m\right\}
	$$ 
	to be a \textit{probability simplex} set
	where $\mathbb{R}_+^m:=\{x=[x_1\;x_2\; \cdots x_m]^T \in \mathbb{R}^m: x_i \geq 0, \; i=1,2,\dots,m\}$.
We assume that the ambiguity set of possible distributions $\cal P$ is given by a finite set of linear inequalities and equalities.\footnote{There are many more forms of the ambiguous distributional set that can be found in the literature; e.g., see \cite{dupacova2009stochastic} and references therein.} That is, $\mathcal{P}$ takes the form of
	\[
	\mathcal{P} := \{p \in S_m: A_0 p = d_0, \;\; A_1 p \leq d_1\}
	\] 
	where $A_0 \in \mathbb{R}^{m_0 \times m}$, $d_0 \in \mathbb{R}^{m_0}$, $A_1 \in \mathbb{R}^{m_1 \times m}$ and $d_1 \in \mathbb{R}^{m_1}.$ 
	Then the worst-case expected log-growth rate $g_\mathcal{P}(K):= \inf_{p \in \mathcal{P}} g_p(K)$ is given by the optimal value of the linear program:
	
	\medskip
	\begin{problem}[Worst-Case Expected Log-Growth] \label{problem: worst case ELG}
	\begin{align*}
		&\min_p \sum_{j = 1}^m p_j \log (1+K^Tx^j)\\
		&{\rm s.t.}\;\; \sum_{i = 1}^n p_j =1, p_j\geq 0, \; j=1,2,\dots,m;\\
		&\;\;\;\;\;\;\; A_0 p = d_0, A_1 p \leq d_1
	\end{align*}
with variables $p=[p_1\;\; p_2\;\; \cdots p_m]^T$.
	\end{problem}

\medskip
\begin{theorem}[Distributional Robust Log-Optimal Portfolio Problem] \label{theorem: distributional robust log-optimal portfolio problem}
Let $\mathcal{P}$ be a polyhedron set for ambiguous return  distributions. The distributional robust log-optimal portfolio problem $\max_{K \in \mathcal{K}} \; \inf_{p \in \mathcal{P}} g_p(K) $ is equivalent to 
	\begin{align*}
	&\max_{K,v, \lambda} \;\;\min_j \left( q(K) + A_0^T v + A_1^T \lambda \right)_j - v^Td_0 - \lambda^T d_1\\
	&{\rm s.t.} \;K \in \mathcal{K}, \; \lambda \succeq 0
\end{align*}
where $q(K):=[q_1(K) \; q_2(K) \; \cdots q_m(K)]^T$ with $q_j(K):=\log(1+K^Tx^j)$; $v \in \mathbb{R}^{m_0}$ and $\lambda \succeq 0$  means that $\lambda \in \mathbb{R}^{m_1}$ is component-wise nonnegative. 
Moreover, the problem above is a concave optimization problem.
\end{theorem}	
	
\begin{proof}
We begin by writing down the Lagrangian
\[
\mathcal{L}(v, \lambda, p) = p^Tq(K) + v^T (A_0 p -d_0) + \lambda^T (A_1 p - d_1) +   1_{S_m}(p).
\]
where $q(K):=[q_1(K) \; q_2(K) \; \cdots q_m(K)]^T$ with $q_j(K):=\log(1+K^Tx^j)$; $v \in \mathbb{R}^{m_0}$ and $\lambda \in R^{m_1}$ with~$\lambda_j \geq 0$ are the {Lagrange dual variables}, and 
 $1_{S_m}(p)$ is an indicator function that represents the probability simplex constraint.  
Minimizing over $p$ yields the Lagrange dual function; i.e.,
\begin{align*}
g(v,\lambda) 
&= \inf_{p \in S_m}  \mathcal{L}(v, \lambda, p) \\
&=  \min_{p \in S_m}  p^T \left( q(K) + A_0^T v + A_1^T \lambda \right)   - v^Td_0 - \lambda^T d_1\\
&=  \min_{p \in S_m}  \sum_{j = 1}^m p_j \left( q(K) + A_0^T v + A_1^T \lambda \right)_j  - v^Td_0 - \lambda^T d_1\\
&\geq  \min_j \left( q(K) + A_0^T v + A_1^T \lambda \right)_j - v^Td_0 - \lambda^T d_1
\end{align*}
where the last inequality holds since $(q(K)  + A_0^T v + A_1^T \lambda)_j \geq \min_i (q(K)  + A_0^T v + A_1^T \lambda)_i$ for all $j=1,2,\dots,m$
and $(z)_j$ is the $j$th entry of the vector $z \in \mathbb{R}^m$. 
The lower bound is attained if $p_j = 1$ for some $j=1,\dots,m$. Hence,
$$
g(v,\lambda)  =  \min_j \left( q(K) + A_0^T v + A_1^T \lambda \right)_j - v^Td_0 - \lambda^T d_1
$$ 
and the dual problem associated with  Problem~\ref{problem: worst case ELG} is given by
\begin{align*}
&\max_{v, \lambda} \;\; \min_j \left( q(K) + A_0^T v + A_1^T \lambda \right)_j - v^Td_0 - \lambda^T d_1\\
&{\rm s.t.} \;\; \lambda \succeq 0.
\end{align*}
With the aids of Slater's condition, it is readily verified that the strong duality holds. Hence, it follows that the dual problem above has the same optimal value as Problem~\ref{problem: worst case ELG}. Therefore,  the distributional robust log-optimal portfolio problem can be written as
	\begin{align*}
&\max_{K,v, \lambda} \;\; \min_j \left( q(K) + A_0^T v + A_1^T \lambda \right)_j - v^Td_0 - \lambda^T d_1\\
&{\rm s.t.} \;K \in \mathcal{K}, \; \lambda \succeq 0 .
\end{align*}
To complete the proof, it suffices to show that the problem above is indeed a concave optimization problem. 
Begin by observing that $q(K) + A_0^Tv + A_1^T \lambda$ are concave in $K,v,\lambda$. Hence, the point-wise minimum is again concave. 
Now note that the remaining terms $v^Td_0$ and $\lambda^Td_1$ are affine in $v$ and~$\lambda$. Hence, it is readily verified that the objective function
$\min_j \left( q(K) + A_0^T v + A_1^T \lambda \right)_j - v^Td_0 - \lambda^T d_1
$ is concave in $K, v, \lambda$. Lastly, we note that $\mathcal{K}$ is a convex set and that the set $\lambda \succeq 0$ forms a convex positive orthant; hence, the intersection is again convex. Therefore, the problem considered has a concave objective with a convex constraint set; hence it is a concave optimization problem.
\end{proof}

\medskip
\begin{remark}\rm
In practice, while the true distributions for the returns may not be available for the trader, it may be partially known in the sense that the trader may ``estimate" or ``forecast" the distribution and form a ``confidence interval" centered at some nominal guess. 
The idea can be characterized by a ``box"-type ambiguous return distributions, which can be shown as a special case of the polyhedron set described above. See the following example to follow.
\end{remark}

\medskip
\begin{example}[Box Ambiguous Distribution Set] \label{example: box ambiguity set}
Assuming that $\mathcal{P}$ is a \textit{box}. That is, for each $p_j$ with $j=1,2,\dots, m$, it has a lower and upper bounds. Specifically, we consider
\[
\mathcal{P} := \{p \in S_m: |p_j - \overline{p}_j| \leq \rho_j, \; j=1,2,\dots,m \}
\]
where $\overline{p} \in S_m$ is the \textit{nominal} distribution\footnote{The nominal distribution can be obtained by estimation using the historical data.} and $\rho_j \geq 0$ is a prespcified \textit{radius} for all $j=1,\dots,m$. 
By setting $\rho:=[\rho_1 \; \rho_2\; \cdots \rho_m]^T$ and $\overline{p}:=[\overline{p}_1\; \overline{p}_2 \; \cdots \overline{p}_m]^T$, the box constraints can be written as $A_1 p \leq d_1$ with
$$
A_1 = \begin{bmatrix}
I \\ -I
\end{bmatrix}; \;\; d_1 :=\begin{bmatrix}
\rho + \overline{p} \\ \rho - \overline{p}  
\end{bmatrix}.
$$
According to Theorem~\ref{theorem: distributional robust log-optimal portfolio problem}, the associated distributional robust log-optimal portfolio problem becomes
	\begin{align*}
	&\max_{K, \lambda} \; \min_j \left( q(K)  + A_1^T \lambda \right)_j - \lambda^T d_1\\
	&{\rm s.t.} \;K \in \mathcal{K}, \; \lambda \succeq 0.
\end{align*}
where $(z)_j$ is the $j$th component of vector $z \in \mathbb{R}^m.$
\end{example}
	
\medskip
\begin{remark}
	As seen later in Section~\ref{section: Illstrative Example}, we shall adopt this box-type distribution set in our empirical studies with historical price data.
\end{remark}

	\section{Supporting Hyperplane Approximation} \label{section: supporting hyperplane approximation}
	In this section, we introduce the supporting hyperplane approximation approach to the distributional robust log-optimal portfolio optimization problem described in Section~\ref{section: Problem Formulation}.
	\subsection{Idea of Supporting Hyperplane Approximation} \label{subsection: Idea of Supporting Hyperplane Approximation}
	The main idea of the supporting hyperplane approximation is as follows.
	Given $x_{\min} \in (-1,0]$ and~$x_{\max}\geq 0$, consider a mapping $f: [x_{\min}, x_{\max}] \to \mathbb{R}$ with 
	$$
	f(x ) = \log(1+ x).
	$$
	It is straightforward to see that such function $f$ is monotonic and concave in $x$. 
	Thus, one way to approximate~$f$ is to introduce a surrogate function consists of $M$ hyperplanes as follows: Partitioning the interval~$[x_{\min}, x_{\max}]$ and get~$M$ partitioned points, say $z_l$ for $l=1,2,\dots,M$, with 
	$
	x_{\min} = z_1 < z_2 < \cdots < z_M =x_{\max}.
	$ Then we take the hyperplanes  of the form
	$$
	h_l(x) := a_l x + b_l
	$$
	where $a_l = 1/(1+z_l)$ and $b_l = \log(1+z_l) - a_l z_l$. Once the points $z_l$ are determined,  the hyperplanes are obtained. See Figure~\ref{fig: supporting hyperplan idea} for an illustration with $M=4$ supporting hyperplanes for $ \log(1+x)$.
	
	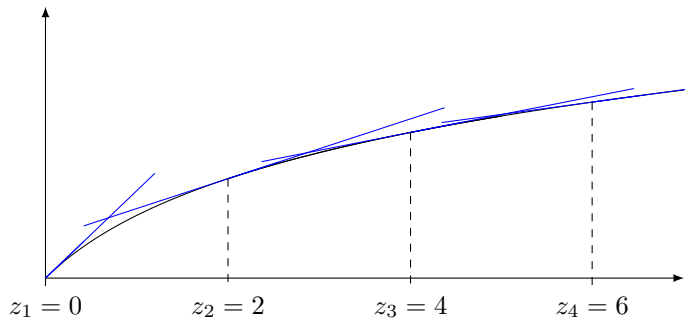
\begin{figure}[h!]
		\centering
		\begin{tikzpicture}[declare function={f(\x)=ln(1 + \x);}, scale =1.2]
			\draw[latex-latex] (0,3) -- (0,0) -- (7,0);
			
			\begin{scope}
				\clip (0,0) rectangle (7,4);
				\draw[tangent=0]  plot[samples=100,domain=0:7] ({\x},{f(\x)});
				\draw[use tangent,blue](-2,0)--(2,0); 
			\end{scope}
			\draw[dashed] (Tang) --(0,-0.1 -| Tang) node[below]{$z_1=0$};

			\begin{scope}
				\clip (0,0) rectangle (7,4);
				\draw[tangent=0]  plot[samples=100,domain=2:5] ({\x},{f(\x)});
				\draw[use tangent,blue](-2,0)--(3,0); 
			\end{scope}
			\draw[dashed] (Tang) --(0,-0.1 -| Tang) node[below]{$z_2=2$};

			\begin{scope}
				\clip (0,0) rectangle (7,4);
				\draw[tangent=0]  plot[samples=100,domain=4:5] ({\x},{f(\x)});
				\draw[use tangent,blue](-2,0)--(3,0); 
			\end{scope}
			\draw[dashed] (Tang) --(0,-0.1 -| Tang) node[below]{$z_3=4$};

			\begin{scope}
				\clip (0,0) rectangle (7,4);
				\draw[tangent=0]  plot[samples=100,domain=5.99:6] ({\x},{f(\x)});
				\draw[use tangent,blue](-2,0)--(3,0); 
			\end{scope}
			\draw[dashed] (Tang) --(0,-0.1 -| Tang) node[below]{$z_4=6$};
			
		\end{tikzpicture} 
		\caption{Approximation for $f(x):=\log(1+x)$  Via $M=4$ Supporting Hyperplanes.}
		\label{fig: supporting hyperplan idea}
	\end{figure}
	
	\subsection{Linear Program Formulation Via Hyperplanes} 
	The idea discussed previously in Section~\ref{subsection: Idea of Supporting Hyperplane Approximation} enables us to approximate the expected log-portfolio optimization problem formulated in Section~\ref{subsection: Expected log-portfolio problem} as a linear program. 
	Specifically,  for $j=1,\dots,m$, we have
		\begin{align*}
	q_j(K) &= \log(1+K^T x^j)\\
		&\approx   \min \left\{ h_1 \left( K^T {x}^j  \right), \dots, h_M \left(  K^T {x}^j  \right) \right\}.
	\end{align*}
%
	where ${x}^j := [{x}_1^j \; \; {x}_2^j \;  \cdots \; {x}_n^j ]^T$
	and
	\begin{align*}
		h_l( K^T {x}^j) 
		&=  a_l (K^T {x}^j )+ b_l \\
		&=a_l \left( \sum_{i=1}^n K_i {x}_i^j \right) + b_l
	\end{align*}
	for $l = 1, 2, \dots, M$, 
	which is  clearly linear in $K$. 
	Now,  for $j=1,2,\dots,m$, we define
	\[
	Z_j := \min \left\{ h_1\left( K^T {x}^j \right), h_2\left( K^T {x}^j \right), \dots, h_M\left( K^T {x}^j \right) \right\}.
	\]
With the aid of the supporting hyperplanes,  we are ready to reformulate the distributional robust log-optimal portfolio problem as  a linear program (LP) that involves the shorting and leverage, survival, and diversified holding constraints as follows.
	

	\medskip
	\begin{problem}[Log-Optimal Portfolio Problem Via Hyperplane Approximation] \label{problem: log-optimal portfolio optimization problem via hyperplane approximation}
	 Given constants $K_{i,\min}<0, K_{i,\max} > 0, L \geq 1$ and $c_i:=0$, we consider the following (robust) linear programming problem:
	\begin{align*}
	&\max_{K,v, \lambda,W} W - v^Td_0 - \lambda^T d_1\\
		&{\rm s. t.}\;\; \lambda \succeq 0\\
		&\hspace{4ex} K_{i,L} \ge 0; \;\; K_{i,S} \le 0, \;\;  i = 1,\dots,n;\\
		&\hspace{4ex} K_{i,\min} \leq K_{i,L}+ K_{i,S} \leq K_{i,\max} \;\; i=1,2,\dots, n;\\
		&\hspace{4ex} \sum_{i = 1}^n {\left| {{K_{i,L}} + {K_{i,S}}} \right|}  \leq L;\\
		&\hspace{4ex} \sum_{i = 1}^n { K_{i,L} |\min \left\{  x_{i,\min }, 0 \right\}} | 
		- \sum\limits_{i = 1}^n {{K_{i,S}}\max \left\{ {0, x_{i,\max} } \right\}}  \le 1\\
		&\hspace{4ex} {Z_j} \le {a_l}\left(  \sum_{i = 1}^n {\left( {{K_{i,L}} + {K_{i,S}}} \right)} {x}_{i}^j \right) + {b_l},\;\;\; j=1,2,\dots,m; \hspace{3mm} l=1,2,\dots,M\\
	&\hspace{4ex} W \leq Z_j + (A_0^T v)_j + (A_1^T \lambda )_j , \; \;\; j=1,2,\dots,m.
	\end{align*}
where $ z \succeq 0$ means that component-wise  non-negativity.
	\end{problem}
	
	Subsequently, it is readily to solve for $K_{i,L}$, $K_{i,S}$, $Z_j$ and $W$ for each $i,j$. 
	The associated optimal solution obtained by the hyperplane approximation-based linear programming is denoted by $K_{h}$ and the corresponding expected log-growth is  $g(K_h)$; see Example~\ref{example: a toy example} for a simple illustration of the theory.
	
	\medskip
	\begin{remark}\rm
		By treating $m$ as the number of scenarios, the formulation above can be viewed as a  \textit{scenario-based} linear program; see \cite{cornuejols2006optimization}. 
	\end{remark}

%
	
\medskip
	 \begin{example}[A Toy Example] \label{example: a toy example}\rm
	Consider a portfolio consisting of $n=2$ assets whose returns are identically distributed with $X_1(k) \in \{x_1^1, x_1^2\}:= \{0.1, -0.25\}$ and~$X_2(k) \in \{ x_2^1, x_2^2\}:=\{-0.1, 0.3\}$. 
		 	The  joint probability is described as follows: 
		 	$$
		 	P(X_1(k) = x_{1}^j, X_2(k) = x_{2}^j) = p_j \in (0,1)
		 	$$ 
		 	for $ j \in \{1,2\}$
and the nominal probabilities $\overline{p}_1 =0.7$ and $\overline{p}_2 = 0.3.$
	We assume that there is no transaction costs $c_i=0$ for $i=1,2$ and the set of ambiguous distribution set is the box
\[
\mathcal{P} :=\left\{ (p_1, p_2) \in S_2: |p_j -\overline{p}_j| \leq \gamma \cdot \overline{p}_j, \; j=1,2 \right\}
\]
with $\gamma:=0.1$ and $S_2:=\{(p_1,p_2): \sum_{j = 1}^2 p_j =1, p_j \geq 0, \; j=1,2\}$. 
With $p=[p_1\;\;p_2]^T$, the constraints $|p_j -\overline{p}_j| \leq \gamma \cdot \overline{p}_j, \; j=1,2$ can be expressed as a matrix inequality $A_1 p \leq d_1$ where
\[
A_1 = \begin{bmatrix}
1 & 0\\
0 & 1\\
-1 & 0\\
0 & -1
\end{bmatrix}; \;\; d_1 = \begin{bmatrix}
(\gamma + 1)\overline{p}_1  \\
(\gamma + 1)\overline{p}_2 \\
(\gamma - 1)\overline{p}_1  \\
(\gamma - 1)\overline{p}_2 
\end{bmatrix}.
\]

Assume that there is no transaction costs, i.e., $c_i = 0$ and  take $M = 3$ hyperplanes,   $K_{i,\min} := 0$, $K_{i,\max}:=1/2$, and~$L:=1$. Then the associated  linear programming problem is given by
	\begin{align*}
	&\max_{K,v, \lambda,W} W  - \lambda^T d_1\\
	&{\rm s. t.}\;\; \lambda \succeq 0\\
	&\hspace{4ex} K_{i,L} \ge 0; \;\; K_{i,S} \le 0, \;\;  i = 1,2;\\
	&\hspace{4ex} 0 \leq K_{i,L}+ K_{i,S} \leq 1/2;\; i=1,2;\\
	&\hspace{4ex} \sum_{i = 1}^ 2{{K_{i,L}} + {K_{i,S}}}  \leq 1;\\
	&\hspace{4ex} {Z_j} \le {a_l}\left(  \sum_{i = 1}^n {\left( {{K_{i,L}} + {K_{i,S}}} \right)} {x}_{i}^j \right) + {b_l},\;\;\; j=1,2; \hspace{3mm} l=1,2,3\\
	&\hspace{4ex} W \leq Z_j  + (A_1^T \lambda )_j , \; \;\; j=1,2.
\end{align*}

Under ambiguity constant $\gamma :=0.1$, 
the associated  portfolio weight obtained by the linear program is given by 
			$$
			K_h= [K_{h,1}\;\; K_{h,2} ]^T= [0.3016\;\; 0.5]
			$$
and the corresponding log-growth is $ g(K_h) = 0.00755$ approximately. Other  approximate portfolio weights under various $\gamma$ are summarized in Table~\ref{table: distributional_robust_example}.
On the other hand, the \textit{true} log-optimal portfolio weight under nominal distribution $\overline{p}$ is $K^* = [0.3698\;\;0.5]^T$ and the associated optimal growth rate is $g^* \approx 0.00761.$ 

\begin{table}[h!]
	\centering
	\caption{Approximate Optimal Portfolio Weights with $M=3$ Hyperplanes}
	\begin{tabular}{|c|c|c|c|}
		\hline
		& $K_{h,1}$ & $K_{h,2}$   & $g(K_h)$  \\
		\hline
		\makecell[c]{ $\gamma =0$ } & 0.3016 &0.5&0.00755 \\
		\hline
		\makecell[c]{$\gamma =0.10$ } & 0.5 &0.5 &0.00740 \\
		\hline
		\makecell[c]{$\gamma =0.15$ } & 0.5 &0.4986 &0.00738 \\
		\hline
		\makecell[c]{$\gamma =0.2$ } & 0.5 &0.4375 &0.00623 \\
		\hline
	\end{tabular}
	\label{table: distributional_robust_example}
\end{table}

\end{example}

%

\medskip										
\begin{remark}\rm
In particular, the number of hyperplanes $M$ can be viewed  as a new \textit{design} variable such that one can reconcile the computational complexity largely. 
To reflect this point, in the sequel, we may sometimes emphasis the dependence on $M$ by denoting that an optimal solution obtained via hyperplane approximation approach as
$
K_{h} := K_{h}(M).
$
\end{remark}

\section{Optimal Number of Hyperplanes} \label{section: optimal number of hyperplanes}
In the previous section, we discussed the hyperplane approximation approach with \textit{ad-hoc} tuning for the number of hyperplanes.  
In this section, we establish an algorithm that is used to select a ``optimal" number of hyperplanes.
That is, given an allowable approximation error, we seek the minimal number of hyperplanes that are needed so that the error is respected.

\subsection{Optimal Number of Hyperplanes}
Given a constant $\varepsilon >0$, let $x \in [x_{\min}, x_{\max}]$ with $x_{\min} > -1$.
Our objective  is to assure that the maximum approximation error of the hyperplane approximation approach is less than or equal to the specified constant~$\varepsilon$.
 That is, 
\[
\sup_{x \in [x_{\min}, x_{\max}]} | \log(1+x) - \min_i \{ a_i x + b_i \}| \leq \varepsilon
\]
where $a_i = 1/(1+x_i)$ and $b_i = \log(1+x_i) - a_i x_i$.\footnote{As seen in Section~\ref{section: supporting hyperplane approximation}, the analysis with the difference $\log(1+x) - \min_i \{a_i x+ b_i\}$ in this section is understood with~$x:=K^Tx^j$.}
To achieve this goal, we proceed as follows.
First we pick $x_0 := x_{\min}$, then compute the first hyperplane $h_0:=\{x: a_0 x+b_0=0\}$ which is tangent at~$x_0$ with the coefficient 
\[
a_0 = \frac{1}{1+x_{0}};\;\; b_0 = \log (1+x_{0})  - a_0 x_0.
\]
Now, let $\chi \in [x_0, x_{\max}]$ be a variable to be determined which is used for finding the second hyperplane $h_1:=\{x: a_1x+b_1=0\}$ with coefficients
\[
a_1 = \frac{1}{1+\chi};\;\; b_1 = \log (1+ \chi)  - a_1 \chi.
\]
Denote the intersection point  between these two hyperplanes $h_0$ and $h_1$ as $(x',y')$ with
\[
x' = \frac{b_1 - b_0}{a_0 - a_1};\;\; y' = a_0 x' + b_0.
\]
It is readily verified that the maximum approximate error happens at the intersection point. Hence, for $\chi >x_0,$ a lengthy but straightforward calculation leads to 
\begin{align}
 |y' - \log (1 + x')| 
	 &=  {{a_0}x' + {b_0} - \log \left( {1 + \frac{{{b_1} - {b_0}}}{{{a_0} - {a_1}}}} \right)} \nonumber \\
	& = \frac{{1 + \chi }}{{\chi  - {x_0}}}\log \frac{{1 + \chi }}{{1 + {x_0}}} - \log \left( {\frac{{1 + \chi }}{{\chi  - {x_0}}}\log \frac{{1 + \chi }}{{1 + {x_0}}}} \right) - 1 \label{eq: e_1(x)} \\ 
	& := {e_1}(\chi ). \nonumber
\end{align} 



	\begin{lemma}[Limiting Behavior of An Approximation Error]
		For any $x > x_0 >-1$,  it follows that
		$
		\lim_{x \to x_0} e_1(x) = 0.
		$
	\end{lemma}

\begin{proof}
First note that for $x > x_0 >-1$, it follows that the ratio $\frac{1+x}{1+x_0} >0$. Hence, using the fact that $\frac{z-1}{z} \leq \log z \leq z-1$ for $z>0$, it implies that
		\begin{align}
		\frac{{x - {x_0}}}{{1 + x}} \leq \log \frac{{1 + x}}{{1 + {x_0}}} \leq \frac{{x - {x_0}}}{{1 + {x_0}}}. \label{ineq: upper and lower bound of log 1+x over 1+x0}
		\end{align}
According to  Equation~(\ref{eq: e_1(x)}), the approximate error is given by
\[
e_1(x) = \frac{ 1 + x }{x  - x_0}\log \frac{1 + x }{1 + x_0} - \log \left( {\frac{1 + x }{x  - x_0}\log \frac{1 + x }{{1 + x_0}}} \right) - 1  . 
\]
With the aid of Inequality~(\ref{ineq: upper and lower bound of log 1+x over 1+x0}), it follows that
		\[
		- \log \frac{{1 + x }}{{1 + x_0}} \leq {e_1}(x ) \leq \frac{{1 + x }}{{1 + {x_0}}} - 1.
		\]
Since $e_1(x)$ is continuous, the limit exists. In addition, note that 
$$
\lim_{x \to x_0} \frac{1+x}{1-x_0}-1=\lim_{x \to x_0} \left( -\log \frac{1+x}{1+x_0} \right) =0.
$$ 
By the Squeeze Theorem; e.g., see \cite{stewart2020calculus}, it follows that~$e_1(x) \to 0$ as $x \to x_0$. 
	
\end{proof}

\medskip
Now, solving $e_1(x) = \varepsilon$ for $x$ and letting its solution to be denoted as $x_1$. 
This procedure can be easily to extend for finding the successive points $\{x_i\}_{i=1}$. 
For instance, given $\varepsilon>0$ and the partition point $x_{i-1}$, to find the  $i$th successive point $x_i$, we proceed as follows: 
We first calculate the approximate error
\begin{align*}
	{e_i}(x) 
	&= \frac{1 + x}{x - x_{i - 1} }\log \frac{1 + x}{ 1 + x_{i - 1} } - \log \left( \frac{1 + x}{x - x_{i - 1}}\log \frac{1 + x}{1 + x_{i - 1}} \right) - 1.
\end{align*}
Then $x:=x_i$ is the solution of $e_i(x) = \varepsilon$. The existence and uniqueness  of $x_i$ is justified by the following lemma.

\begin{lemma}[Monotonic Approximate Error] \label{lemma: Monotoic Approximate Error}
	For any $x$ satisfying $x> x_0 > -1$, the approximate error 
\begin{align*}
	{e_i}(x) 
	&= \frac{1 + x}{x - x_{i - 1} }\log \frac{1 + x}{ 1 + x_{i - 1} } - \log \left( \frac{1 + x}{x - x_{i - 1}}\log \frac{1 + x}{1 + x_{i - 1}} \right) - 1
\end{align*} for $i=1,2,\dots, M$
 is strictly increasing in~$x$.  
\end{lemma}

\begin{proof}
Let $x$ satisfying $x> x_0 > -1$. It suffices to show that $e_1(x)$ is strictly increasing since~$e_i(x)$ takes the an almost identical form of $e_1$.
In particular, note that
\[
e_1 (x)  = \frac{{1 + x }}{{x  - {x_0}}}\log \frac{{1 + x }}{{1 + {x_0}}} - \log \left( {\frac{{1 + x }}{{x  - {x_0}}}\log \frac{{1 + x }}{1 + x_0}} \right) - 1 .
\]
Taking derivative of $e_1(x)$ with respect to $x$, we obtain
$$
\frac{d}{{dx}}e_1 (x)  = f(x) - g(x)
$$
where
\begin{align*}
	f(x) & := \frac{{x + x_0^{} + 2}}{{\left( {x + 1} \right)\left( {x - x_0^{}} \right)}};\\
	g(x) & := {\frac{{\left( x_0 + 1 \right)\log \left( {\frac{ 1+x}{ 1+x_0 }} \right)}}{{{{\left( {x - x_0} \right)}^2}}} + \frac{1}{( 1+x  )\log \left( \frac{1+x}{1 + x_0 } \right)}}.
\end{align*}
To prove that $e_1(x)$ is strictly increasing, it suffices to show that  $f(x) - g(x) > 0$ for all $x$.  Suppose this is not the case by assuming that there exists $x>x_0>-1$ such that $f(x) - g(x) \le 0$.
Take such an $x$. 
Now substituting $f(x)$ and $g(x)$ back into the inequality $f(x)-g(x) \leq 0$, a lengthy but straightforward calculation yields
\begin{align}
	\left[ {\left( x_0 + 1 \right)\left( {x + 1} \right)b - \left( {x - x_0^{}} \right)\left( {x + x_0^{} + 2} \right)} \right]b + {\left( {x - x_0^{}} \right)^2} \ge 0 \label{ineq: de_x >=0}
\end{align}
where $b := \log\left( (1+x)/(1+x_0) \right)>0$ since $x > x_0 > -1$. 
Note that the second term ${\left( {x - x_0^{}} \right)^2}>0$, Inequality (\ref{ineq: de_x >=0}) holds if  $ {\left( x_0 + 1 \right)\left( {x + 1} \right)b - \left( {x - x_0} \right)\left( {x + x_0+ 2} \right)} \geq 0$.
Equivalently,
\begin{align}
	b \ge \frac{{\left( {x - x_0^{}} \right)\left( {x + x_0^{} + 2} \right)}}{{\left( {x_0+ 1} \right)\left( {x + 1} \right)}}. \label{ineq: b inequality}
\end{align}
Note that for $x > x_0 >-1$, using the fact that $\log (1+x) < x$, it follows that
$
b < \frac{x - x_0}{1 + x_0}.
$
Hence, in combination with Inequality~(\ref{ineq: b inequality}), we  conclude that
\[\frac{{x - {x_0}}}{{1 + {x_0}}} > b \ge \frac{{\left( {x - x_0^{}} \right)\left( {x + x_0^{} + 2} \right)}}{{\left( {x_0^{} + 1} \right)\left( {x + 1} \right)}}
\]
However, the inequality on left-hand side and right-hand side yields $0 > x_0^{} + 1$ which contradicts to $x_0 > -1.$  Therefore, $f(x) - g(x) >0$, which proves that $e_1(x)$ is strictly increasing.
\end{proof}

\medskip			
\begin{remark}[Existence and Uniqueness of Partition Points]\rm
	Given a specified approximate error $\varepsilon >0$, Lemma~\ref{lemma: Monotoic Approximate Error} tells us that
 	  the corresponding point $x$ that solves $e_i(x) = \varepsilon$  can be determined uniquely since~$e_i(x)$ is strictly increasing.  
\end{remark}

\medskip
\subsection{Determination of The Successive Partition Points}
Given the $i$th partition point $x_i$ and $\varepsilon>0$, it is possible to determine the successive partition point~$x_{i+1}$ by an iteration.

\begin{lemma}[Iterative Formula for Successive Partition Point]  \label{lemma: iterative formula for next x_i}
	Given $\varepsilon>0$ and the partition point~$x_{i} >x_0>-1$ that satisfies $e_i(x_i)=\varepsilon$. The successive partition point
	$x_{i+1}$ that also satisfies $e_{i+1}(x_{i+1}) = \varepsilon$ can be written as
	\[
x_{i+1} = (1 + \alpha^*) x_{i} + \alpha^*  
	\]
	where $\alpha=\alpha_\epsilon >0$ solves
	$
	\frac{{1 + \alpha }}{\alpha }\log \left( {1 + \alpha } \right) = \beta^*
	$
	with $\beta^* $ obtained as a solution of
	$
	\beta - \log \beta - 1 = \varepsilon.
	$
\end{lemma}

\begin{proof}
Fix $\varepsilon>0$. Take  $x_i$ such that $e_i(x_i)=\varepsilon$ and $x_{i+1}$ such that $e_{i+1}(x_{i+1}) = \varepsilon$. Now consider an auxiliary function 
$
f(\beta) := \beta - \log \beta - 1 - \varepsilon
$
and we choose 
\begin{align*}
	\beta 
	&:=  \frac{ 1 + x_{i+1}}{x_{i+1} - x_{i } } \log \frac{ 1 + x_{i+1}}{1 + x_{i } }.
\end{align*}
Then it solves $f(\beta)=0$ since it satisfies the approximate error function $e_{i+1}(x_{i+1}) = \varepsilon$. 
Furthermore, by Lemma~\ref{lemma: Monotoic Approximate Error}, $e_{i+1}(x)$ is strictly increasing, which implies that $f(\beta)$ is also strictly increasing. 
Hence,~$\beta$ is uniquely determined, call such solution as $\beta := \beta^*.$ 
Now, observe that
\begin{align*}
	\beta^*   
	&= \frac{ 1 + x_{i+1}}{ x_{i+1} - x_{i }}\log \left( 1 + \frac{ x_{i+1} - x_{i}}{1 + x_{i} } \right)\\[1ex]
	&= \frac{ \frac{1 + x_{i+1}}{1+x_i}}{  \frac{x_{i+1} - x_{i }}{1+x_i} }\log \left( 1 + \frac{ x_{i+1} - x_{i}}{1 + x_{i} } \right)\\[2ex]
	&= \frac{ 1+ \frac{x_{i+1} - x_i }{1+x_i}}{  \frac{x_{i+1} - x_{i }}{1+x_i} }\log \left( 1 + \frac{ x_{i+1} - x_{i}}{1 + x_{i} } \right)\\[1ex]
	&= \frac{{1 + \alpha }}{\alpha }\log \left( {1 + \alpha } \right)
\end{align*}
where
\begin{align}
	\alpha  &=\frac{x_{i+1} - x_i}{ 1+x_i}:= \alpha^* \label{eq: alpha}.
\end{align}
Note that $\alpha^* >0$ since $x_{i+1}> x_{i} > -1$. 
Then, by Equation~(\ref{eq: alpha}), it is readily verified  that 
$$
x_{i+1} = (1 + \alpha^*) x_{i} + \alpha^*. 
$$
\end{proof} 

\begin{remark}\rm
	Having obtained the iterative formula in Lemma~\ref{lemma: iterative formula for next x_i}, it enables us to readily determine the \textit{minimum} number of supporting hyperplanes that assures an allowable approximate error. This is summarized in Algorithm~\ref{alg::OptNum_Hypers} to follow.
\end{remark}


\begin{algorithm}[h!]
	\caption{$\varepsilon$-Optimal Number of Hyperplanes}
	\label{alg::OptNum_Hypers}
	\begin{algorithmic}[1]
		\Require
		Given $\varepsilon >0$ and pick $x_{\min}$, $ x_{\max}$ with~$x_{\min} >-1$ and $x_{\max}\geq 0 \geq x_{\min}$;
		\Ensure
		optimal partition points $\{x_i\}_{i=0}$    
		\State Initial Step: Set $x_0 := x_{\min}$. Compute $
		a_0 = \frac{1}{1+x_{0}}$ and $b_0 = \log (1+x_{0})  - a_0 x_0$ to get the first hyperplane $a_0 x + b_0$;
		\Repeat
		\State Use Lemma~\ref{lemma: iterative formula for next x_i} to find $x_i$ by solving $e_i(x_i) = \varepsilon$;
		\State Use $x_i$ to compute $a_i=1/(1+x_i)$ and $b_i=\log(1+x_i) - a_i x_i$  ;
		\State Get the associated hyperplane $a_i x + b_i$;
		\Until{($x_i \geq x_{\max}$)}
	\end{algorithmic}
\end{algorithm}

\medskip
	\begin{theorem}[Optimal Number of Hyperplanes] \label{theorem: optimal number of hyperplanes}
		Given $\varepsilon >0$, $x_{\max}$ and $x_{\min}$ with~$x_{\max}> x_{\min} >-1$,  Algorithm~\ref{alg::OptNum_Hypers} leads to the minimum number of supporting hyperplanes satisfying the approximation error $e_i(x) \leq \varepsilon$ for all $x \in [x_{i-1}, x_{i}] \subset [x_{\min}, x_{\max}]$  and $i=1,2,\dots,M$.
	\end{theorem}

\begin{proof}
Suppose that Algorithm~\ref{alg::OptNum_Hypers} leads to a set of finite points $\{x_0, x_1, \dots,x_{M}\}$ 
satisfying 
$$
x_{\min}= x_0 < x_1 < \dots < x_{M} = x_{\max}
$$
which partitions the interval $[x_{\min},x_{\max}].$
The corresponding approximation error is $e_i(x) \leq \varepsilon$ for all~$x \in [x_{i-1},x_{i}]$ and all $i$. Hence, $\sup_i e_i = \varepsilon$. 
We now proceed a proof by contradiction. Suppose that there is an alternative hyperplane generating algorithm that generates a less number of hyperplanes than that generated by Algorithm~\ref{alg::OptNum_Hypers}; i.e., it leads to a new set of finite points 
$\{z_0, z_1, \dots, z_L\}$ satisfying 
$$
x_{\min}= z_0 < z_1 < \dots < z_{L} = x_{\max}$$ with $L<M$ which also partitions the interval $[x_{\min},x_{\max}]$ and the associated approximation error is~$e_j(z) \leq \varepsilon$ for all $z \in [z_{j-1}, z_j]$ and all $j$. Hence, $\sup_j z_j \leq \varepsilon$.
By Lemma~\ref{lemma: Monotoic Approximate Error}, it follows that the approximation error function $e_i(x)$ is strictly increasing. In addition,  $e_i(x_i) = \varepsilon$ for all $i$.
Hence, fix $x_i$, there exist a point $z_j$ for some $j \in \{1,2,\dots,L\}$ such that $z_j > x_i$ and the associated approximation error
$
e_j(z_j) > e_i(x_i) = \varepsilon.
$ 
Hence, it follows that 
$
\sup_j e_j > \varepsilon
$
which is a contradiction to the fact that $\sup_j e_j \leq \varepsilon$ 
%
Therefore, Algorithm~\ref{alg::OptNum_Hypers} leads to the minimum hyperplanes satisfying the approximate error $e_i(x) = \varepsilon$ for all $x \in [x_{\min}, x_{\max}]$.  
\end{proof}
			
										
\medskip										
\begin{remark}\rm
	If there is no ambiguity and the joint probabilities $p$ are known in advance, then Lemma~\ref{lemma: Monotoic Approximate Error} and Theorem~\ref{theorem: optimal number of hyperplanes} imply that the log-optimal portfolio $K^*$ and the approximate optimal portfolio $K_h$ can be made as close as possible as the number $M$ of hyperplanes increases.
\end{remark}										

\section{Empirical Studies Using Historical  Data} \label{section: Illstrative Example}
	To illustrate the theory, we solve a distributional robust log-optimal portfolio problem via the hyperplane approximation approach.
	
\medskip	
\begin{example}[Experiments with Historical Stock Prices] \rm  \label{example: numerical experiments}
We consider a portfolio consisting of the top $n=15$ companies in the S\&P500 Index covering the period from January~2,~2021 to June~30,~2021 ($N=124$ trading days) as for in-sample optimization; see Table~\ref{table: top 15 companies} for the symbols of these companies. 
An additional six months, starting from July 01, 2021, to December~31, 2021,  are also included for the out-of-sample test.
The data are obtained from the Yahoo! Finance website.
Figures~\ref{fig:15stockprices} and \ref{fig:15stockreturns}  show  the stock prices and corresponding rate of returns for the 15 companies in the period considered.
In the figures, the trajectories with blue color are used for in-sample optimization and that with red color are used later for the out-of-sample test.

Consistent with Section~\ref{subsection: polyhedron ambigious return distributions}, we assume that the set of ambiguous return distributions is the box
	\[
	\mathcal{P} :=\left\{ p \in \mathbb{R}_+^m: |p_j -\overline{p}_j| \leq \gamma \cdot \overline{p}_j, \;\sum_{j = 1}^m p_j = 1, \;\; p_j \geq 0 \; {\rm for }\; j=1,2,\dots,m\right\}
	\]
	with $\gamma \in (0,1)$. 
	Recalling the analysis in Example~\ref{example: box ambiguity set}, it follows that
	$$
	A_1 = \begin{bmatrix}
		I_{m\times m} \\ -I_{m\times m}
	\end{bmatrix}; \;\; d_1 :=\begin{bmatrix}
		(\gamma+ 1)\overline{p} \\ (\gamma  - 1)\overline{p}   
	\end{bmatrix} \in \mathbb{R}^{(2m) \times 1}.
	$$
	The corresponding (approximate) log-optimal portfolio problem to be solved is 
	\begin{align*}
		&\max_{K,v, \lambda,W} W  - \lambda^T d_1\\
		&{\rm s. t.} \;\; \lambda \succeq 0\\
		&\hspace{4ex} K_{i,L} \ge 0; \;\; K_{i,S} \le 0, \;\;  i = 1,\dots,n;\\
		&\hspace{4ex} K_{i,\min} \leq K_{i,L}+ K_{i,S} \leq K_{i,\max} \;\; i=1,2,\dots, n;\\
		&\hspace{4ex} \sum_{i = 1}^n {\left| {{K_{i,L}} + {K_{i,S}}} \right|}  \leq L;\\
		&\hspace{4ex} \sum_{i = 1}^n { K_{i,L} |\min \left\{  x_{i,\min }, 0 \right\}} | 
		- \sum\limits_{i = 1}^n {{K_{i,S}}\max \left\{ {0, x_{i,\max} } \right\}}  \le 1\\
		&\hspace{4ex}{Z_j} \le {a_l}\left(  \sum_{i = 1}^n {\left( {{K_{i,L}} + {K_{i,S}}} \right)} {x}_{i}^j \right) + {b_l},\;\;\; j=1,2,\dots,m; \hspace{3mm} l=1,2,\dots,M\\
		&\hspace{4ex} W \leq Z_j + (A_1^T \lambda )_j , \; \;\; j=1,2,\dots,m.
	\end{align*}
	which is a linear program.
In the sequel, we take ambiguity constant $\gamma:=0.1$. 
On a $3.2$ GHz laptop with~$16$~GB RAM, the distribution log-portfolio optimization problem solved by the hyperplane approximation approach is less than $1$ second.
	
	\begin{table}[h!]
		\centering
		\caption{The Top 15 companies on the S\&P 500 index.}
		\begin{tabular}{|c|c|c|}
			\hline
			\# & Company & Symbol \\
			\hline
			1 & Apple Inc. & AAPL \\
			\hline
			2 &  Microsoft Corporation& MSFT  \\
			\hline
			3&  Amazon.com Inc. & AMZN  \\
			\hline
			4&  Tesla Inc.& TSLA  \\
			\hline
			5&  Alphabet Inc. Class A&GOOGL  \\
			\hline
			6& Alphabet Inc. Class C& GOOG \\
			\hline
			7&  Meta Platforms Inc. Class A& FB \\
			\hline
			8&  NVIDIA Corporation& NVDA \\
			\hline
			9&  Berkshire Hathaway Inc. Class B& BRK.B  \\
			\hline
			10&  JPMorgan Chase \& Co.&  JPM\\
			\hline
			11& Johnson \& Johnson & JNJ \\
			\hline
			12&  UnitedHealth Group Incorporated &  UNH\\
			\hline
			13&  	Procter \& Gamble Company& PG  \\
			\hline
			14&  Home Depot Inc.&  HD \\
			\hline
			15&  Visa Inc. Class A & V \\
			\hline
		\end{tabular}
		
		\label{table: top 15 companies}
	\end{table}
	
	\begin{figure}[h!]
		\centering
		\includegraphics[width=0.8\linewidth]{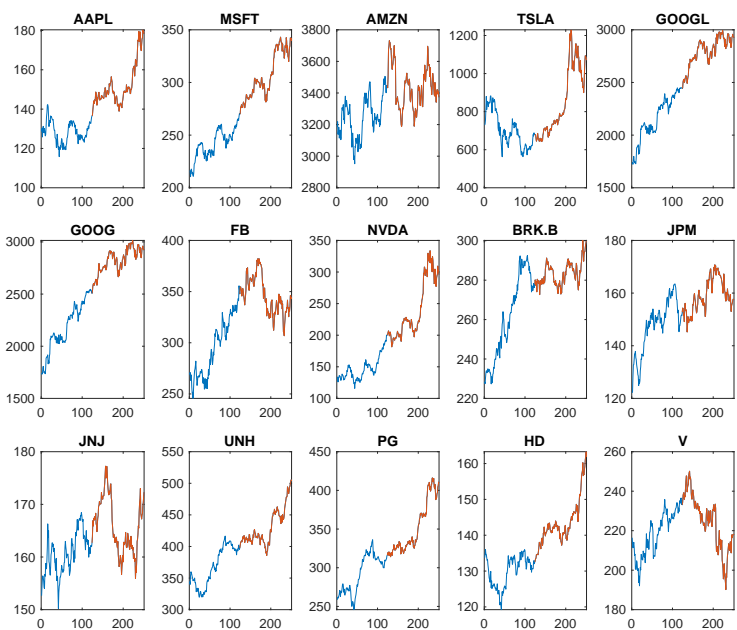}
		\caption{Stock Prices for the Top 15 Companies in the S\&P 500 Index.}
		\label{fig:15stockprices}
	\end{figure}
	
	\begin{figure}[h!]
		\centering
		\includegraphics[width=0.8\linewidth]{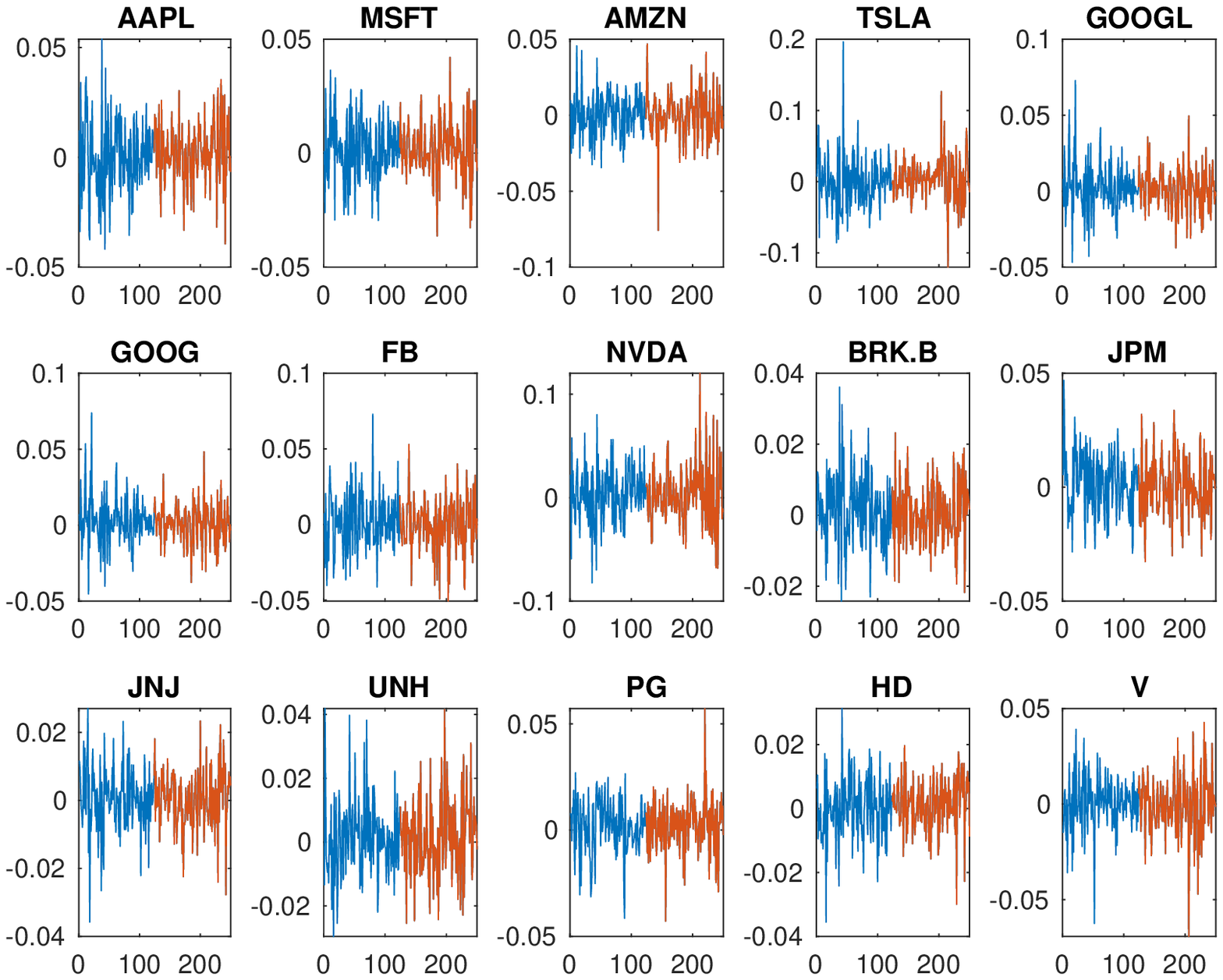}
		\caption{Rate of Returns for the Top 15 Companies in the S\&P 500 Index}
		\label{fig:15stockreturns}
	\end{figure}

\medskip
\textit{In-Sample Optimization.}
We estimate the joint probability $p_j = 1/m$ for all $j=1,2,\dots,m=123$ and take leverage constant $L=2$, and the constants for holding constraints  $K_{i,\max} = L/n = 2/15$ and $K_{i,\min}=0$ for all $i=1,2,\dots,15$. 
Hence, the trades are long-only but leverageable up to twice of investor's wealth.\footnote{It is well-known that the portfolio optimization in practice suffers greatly from the estimation error. In~\cite{jagannathan2003risk}, they proved that constraining portfolio weights to be nonnegative can reduce the risk in estimated optimal portfolios; see also \cite{demiguel2009generalized} for a generalization along the line of research by constraining portfolio norms.
}
In addition, to invoke our supporting hyperplane approach, take~$\varepsilon:=0.01$ for the allowable approximate error. Then by Algorithm~\ref{alg::OptNum_Hypers}, it yields $M=10$ hyperplanes to be used.
With these supporting hyperplanes, we solve the corresponding linear programming\footnote{There are many efficient solution packages or software routines that solves a linear program; e.g., MATLAB {linprog} function or using a modeling framework CVX; see \cite{grant2014cvx, diamond2016cvxpy} for further details.  }  Problem~\ref{problem: log-optimal portfolio optimization problem via hyperplane approximation}
and obtain the associated robust portfolio weighting vector~$K_h:=[K_{h,1}, \cdots, K_{h,n}]^T$  given by 
\begin{align*}
&K_{h,2}=K_{h,5}  = K_{6,h} =\cdots= K_{10,h} =K_{12,h}= K_{13,h} \approx 0.133;\\
& K_{h,1} =	K_{h,4} =	K_{11,h} =		K_{14,h} = 0;\\
&K_{h,3} = 0.0658;\\
&K_{h,15} = 0.1057.
\end{align*}
The associated logarithmic growth rate is $g_h(K_h) \approx 0.0029$.
On the other hand, we also solve the true log-optimal portfolio weight $K^*:=[K_1^* \;\; K_2^*\;\; \cdots \;\; K_n^*]^T$ and obtain a similar pattern as seen in~$K_h$ obtained previously. 
Specifically,
  \begin{align*}
		&K _1^*=K_2^*=K_3^*=K_5^*=\cdots =K_{10}^* = K_{12}^* = K_{13}^*=K_{15}^*  \approx 0.133 ;\\
		&K_{4}^* \approx 0.0016;\;\; K_{11}^* = 0.0819;\;\; K_{14}^* \approx 0.0041
 \end{align*}
and the associated optimal log-growth is $g^* \approx 0.0029$, which, as expected, is  very close to the~$g_h(K_h)$.  Figure~\ref{fig:portfolioweightscomparison} visualizes the portfolio weights $K_h$ and $K^*$ obtained above in a bar plot.

\begin{figure}[h!]
	\centering
	\includegraphics[width=0.8\linewidth]{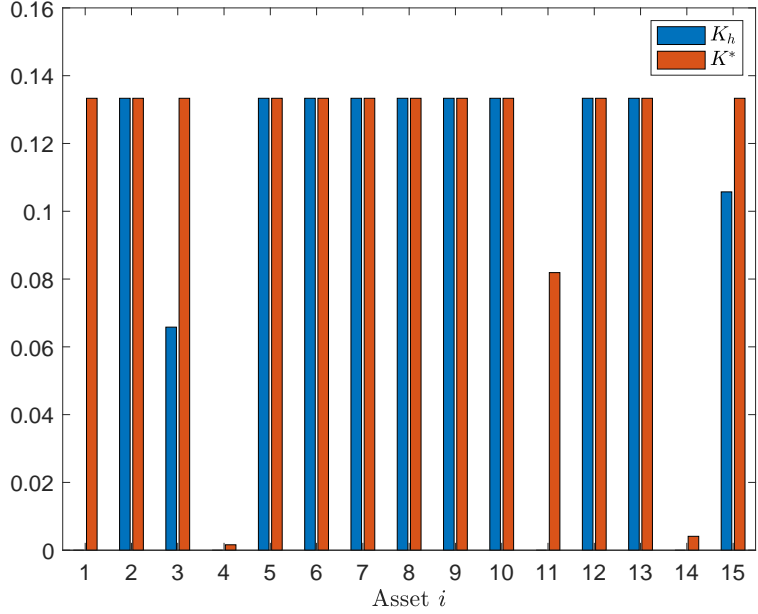}
	\caption{The Approximate Optimal Weight $K_h$ and the True Log-Optimal Weight $K^*$.}
	\label{fig:portfolioweightscomparison}
\end{figure}

\medskip
\textit{Metrics Relative to a Benchmark.}
With the transaction costs $c_i >0$ for $i=1,\dots,n$, recalling that the \textit{portfolio realized return} in period $k$ is defined as
\[
R^p(k) := \frac{V(k+1) - V(k)}{V(k)} .
\]
The (realized) cumulative return up to stage $k=N$ is given by $(V(N) - V(0) )/ V(0)$ and the log-growth rate is the logarithm of the realized cumulative return; i.e., $\log (V(N)/V(0))$.
The \textit{excess return} denoted 
$$R^p(k) := R^p(k) - r_f$$ where~$r_f$ is the \textit{risk-free rate}.
The \textit{realized (per-period) Sharpe ratio}, denoted by $SR$, of the portfolio is the average of the excess returns $\overline{R^p}$ over the standard deviation of the excess returns~$\sigma$; i.e., 
$$
SR:= \frac{\overline{R^p} - r_f}{ \sigma}
$$
and the \textit{N-period realized Sharpe ratio} can be approximated by $\sqrt{N}\cdot SR$; see \cite{lo2002statistics}.
Lastly, other than standard deviation, to scrutinize the downside risks over multi-period trading performance, we include the \textit{maximum percentage drawdown} as our risk metrics.
\[
d^*:= \max_{0\leq \ell < k \leq N} \frac{V(\ell) - V(k)}{V(\ell)}.
\]
Henceforth, we assume that the transaction costs is $c_i := 0.01\%$ of the trade value for each Asset~$i$ and (per-period) risk-free rate is given by $r_f(k):=0.01/N$ for all periods $k=0,1, \dots, N-1$. 

\medskip
\textit{In-Sample Trading Performance.}
As described previously, the in-sample trading involves the first six months with total~$N=124$ trading days.
Figure~\ref{fig:insample} depicts the in-sample trading performance in terms of the account value trajectories using~$K_h$ and~$K^*$. 
Consistent with our theory, a similar pattern of the two account value trajectories is seen in the figure.
In this specific example, we see that the approximate optimal weight $K_h$ even leads to superior performance to the true optimal weight $K^*.$ 
Additionally, we summarize some performance benchmarks mentioned previously in Table~\ref{table: in-sample trading performance summary}.

\begin{figure}[h!]
	\centering
	\includegraphics[width=0.7\linewidth]{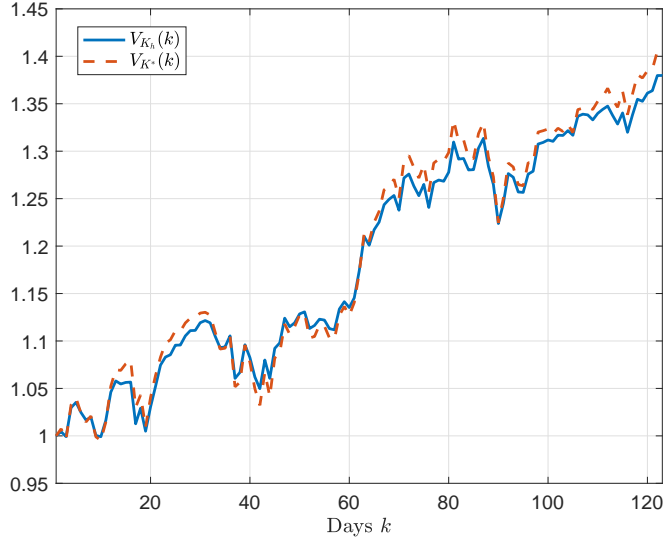}
	\caption{In-Sample Trading Performance}
	\label{fig:insample}
\end{figure}

\begin{table}[h!]
	\centering
	\caption{In-Sample Trading Performance Summary}
	\begin{tabular}{|c|c|c|c|}
		\hline
		& $K_h$ & $K^*$     \\
		\hline
		\makecell[c]{Average Excess Return: $\overline{R}^p - r_f$ } & 0.00262 &0.00282 \\
		\hline
		\makecell[c]{Standard Deviation of Excess Return: $\sigma = {\rm std}({R}^p(k) - r_f)$ } & 0.0139 &0.0168 \\
		\hline
		\makecell[c]{Sharpe Ratio: $ \sqrt{N} \cdot (\overline{R}^p-r_f )/ \sigma$ } & 2.0841 &1.8735 \\
		\hline
		\makecell[c]{Cumulative Return: 
			$\frac{ V(N) - V(0) }{ V(0)}$} & 37.82\% &40.71\% \\
		\hline
		\makecell[c]{ Log-Growth Rate of Wealth: $\log \frac{V(N)}{V(0)}$}& 0.3208  &0.3415  \\
		\hline
		\makecell[c]{Maximum Percentage Drawdown:\\
			$
			\max_{0\leq \ell < k \leq N} \frac{ V(\ell) - V(k)}{ V(\ell)}$ }
		& 6.86\% & 8.67\% \\	
		\hline
	\end{tabular}
\label{table: in-sample trading performance summary}
\end{table}

\medskip
\textit{Out-of-Sample Trading Performance.}
We now carry out an out-of-sample test by considering an additional sixty months within the period from July 01, 2021, to December 31, 2021. 
The total trading days are $N=126$ days. 
Again, we assume that the transaction costs $c_i := 0.01\%$ of the trade value for each Asset $i$.
Beginning with~$V(0)= \$1$, we compare the account value using the approximate weight $K_h$ and log-optimal weight~$K^*.$ Figure~\ref{fig:outofsample} shows the account value trajectories for $V_{K_h}(k)$ and $V_{K^*}(k)$ within the out-of-sample horizon. 
Consistent with our theory, while there is an ambiguity in return distributions, the similar patterns of the account value trajectories using~$K^*$ and $K_h$, respectively, are similar enough.
Specifically, after $126$ trading days, we see that the terminal account values are $V_{K_h}(N) \approx 1.1817$ and $V_{K^*}(N) \approx 1.2467$.
This shows that our hyperplane approximation approach is indeed a competitive alternative for solving a distributional robust log-optimal portfolio problem.
 Table~\ref{table: trading performance summary}  reports some other performance benchmarks.

\begin{figure}[h!]
	\centering
	\includegraphics[width=0.7\linewidth]{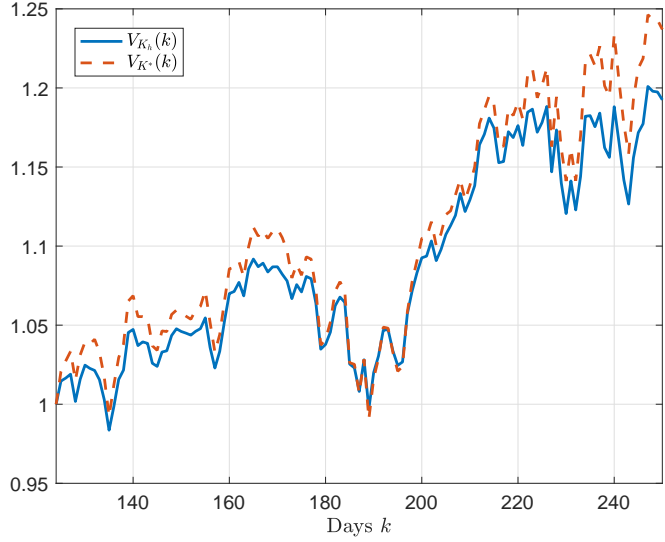}
	\caption{Out-of-Sample Trading Performance}
	\label{fig:outofsample}
\end{figure}

\medskip
\begin{remark}\rm
	In Example~\ref{example: numerical experiments}, we see that only $10$ hyperplanes are needed for assuring approximate error to be less than $\varepsilon=0.01$ when there is no ambiguity on return distributions. 
This idea enables us to form the linear program; i.e., Problem~\ref{problem: log-optimal portfolio optimization problem via hyperplane approximation}, and solve it efficiently.
Said another way, the hyperplane approximation approach shows great potential as an alternative way for solving the distributional log-optimal portfolio problem.
\end{remark}


%

%

\begin{table}[h!]
	\centering
\caption{Out-of-Sample Trading Performance Summary}
\begin{tabular}{|c|c|c|c|}
	\hline
	 & $K_h$ & $K^*$     \\
	 \hline
	 	 \makecell[c]{Average Excess Return: $\overline{R}^p - r_f$ } & 0.00141 &0.00173 \\
		 \hline
	\makecell[c]{Standard Deviation of Excess Return: $\sigma = {\rm std}({R}^p(k) - r_f)$ } & 0.0134 &0.0158 \\
			 \hline
	\makecell[c]{Sharpe Ratio: $ \sqrt{N} \cdot (\overline{R}^p-r_f )/ \sigma$ } & 1.1817 &1.2467 \\
	\hline
	 \makecell[c]{Cumulative Return: 
	 	$\frac{ V(N) - V(0) }{ V(0)}$} & 19.24\% &23.70\% \\
	\hline
	\makecell[c]{ Log-Growth Rate of Wealth: $\log \frac{V(N)}{V(0)}$}& 0.1760  &0.2127  \\
	\hline
	 \makecell[c]{Maximum Percentage Drawdown:\\
	 	$
	 \max_{0\leq \ell < k \leq N} \frac{ V(\ell) - V(k)}{ V(\ell)}$ }
	 & 8.68\% & 10.91\% \\	
	\hline
\end{tabular}

\label{table: trading performance summary}
\end{table}
		\end{example}

		
		\section{Conclusions and Future Work} \label{section: conclusions and future work}
		In this paper,  we provide a  supporting hyperplane approximation approach for solving a class of distributional robust log-optimal portfolio selection problems under a polyhedron set for ambiguous return distributions. 
	Our framework are flexible enough to allow various practical trading requirements such as \textit{transaction costs}, \textit{leverage and shorting}, \textit{survival trades}, and \textit{diversification considerations}.
	With the aid of these supporting hyperplanes, we reformulated the distributional robust optimization problem as a linear program and solve it in a very efficient way. 
	We also proved that the approximate solution obtained by linear programming can be arbitrary close to the true log-optimal portfolio.
	Extra flexibility for reducing the computational complexity by selecting the optimal number of supporting hyperplanes for approximation is also studied.
		To pursue further along the line of this research, two interesting directions are listed as follows.
		
		\subsection{Moment Ambiguity Set}
One interesting direction would be to replace the box ambiguity with some moment ambiguity. 
		 The reason is that the investors are typically able to obtain the estimated mean and covariance matrix subject to some bounds on the estimation errors; see \cite{delage2010distributionally}.
		 To this end, one might consider an ambiguity set with estimated mean $\mu$ and a positive definite covariance matrix~$\Sigma$ of return~$r\in \mathbb{R}^n$. That is,
		 \[
		 \mathcal{P} := \left\{p \in S_m: \mathbb{E}_p[r - \mu]^T\Sigma^{-1} \mathbb{E}_p[r-\mu] \leq \rho_1, \; \mathbb{E}_p[(r-\mu)(r-\mu)^T] \leq \rho_2\right\}
		 \] 
		 for some constants $\rho_1, \rho_2$ where $S_m$ is the probability simplex as defined previously.
		 As mentioned in \cite{sun2018distributional, rujeerapaiboon2016robust}, the distributional robust log-optimal portfolio problem with the moment-based ambiguity constraint above is equivalent to a semidefinite programming problem (SDP). 
		 However,  via a similar hyperplane or a certain quadratic approximation approach, we envision that  such a problem may be approximated by a  more computationally tractable  second-order cone program or even a simpler quadratic program.

		\subsection{Incorporating with Drawdown Risk} \label{section: Control of Drawdown Risk}
				 Another interesting direction would be incorporating some \textit{drawdown} risk constraints into our optimization formulation. 
		In practice, control of drawdown is arguably the most important risk management task for a trader or fund manager.
		According to \cite{chekhlov2004portfolio}, it is unlikely that a particular one would tolerate more than~50\% drawdown in the account. 
	To address this, we provide a way to incorporate our framework to involve a consideration regarding drawdown~risk.

		Given any sample path $\{V(k): k \geq 0\}$, the \textit{maximum percentage drawdown} is defined as 
		$$
		d_K^* := \max_{0\leq \ell < k \leq N} \frac{ V_K(\ell) - V_K(k)}{ V_K(\ell)}.
		$$
		In the sequel, we shall sometimes drop the word ``percentage" in reference to this quantity.
		Take 
		$$
		D_K:=1-d_K^*,
		$$
		which represents the \textit{complementary drawdown}.  
		Then,  given any $\delta \in (0,1)$,  the drawdown constraint $d_K^* \leq \delta$ with probability one is equivalent to~$D_K \geq 1-\delta$. Taking the logarithm on both sides, we 
		form a \textit{surrogate} expected maximum drawdown constraint
		\[
		\mathbb{E}[\log D_K] \geq \log(1-\delta).
		\]
		The surrogate is useful since it forms a convex constraint set and hence facilitates the optimization. This result is stated in the following lemma.
		
		\begin{lemma}[Convex Drawdown Surrogate] \label{lemma: convex drawdown surrogate}
				Fix $\delta \in (0,1)$, we have the following results. 
				
				$(i)$ The following identity 
				\[
				\mathbb{E}[\log D_K] = \mathbb{E}\left[ \min_{0 \leq \ell < k \leq  N} \sum_{i = \ell}^{k-1} \log (1+K^TX(i)) \right]
				\] holds.
				
				$(ii)$ The associated constraint set
	$$
				\mathcal{D}_{K, \delta}:=\{K \in \mathbb{R}^n: \mathbb{E}[\log D_K] \geq \log(1-\delta)\}
			$$
				is convex.
			\end{lemma}
		
\begin{proof}
			The results were originally stated in~\cite{hsieh2015kelly}. However, for the sake of completeness, a full proof is provided here.
			To prove part~$(i)$, we fix $\delta \in (0,1)$ and observe that
			\begin{align*}
					\mathbb{E}[\log D_K]
					&= \mathbb{E}[\log (1- d^*) ]\\
					&=\mathbb{E}\left[ \log \left(1- \max_{0\leq \ell < k \leq N} \frac{V(\ell) - V(k)}{V(\ell)} \right) \right]\\
					&=\mathbb{E}\left[ \log \left(\min_{0\leq \ell < k \leq N}   \frac{V(k)}{V(\ell)} \right) \right]\\
					&=\mathbb{E}\left[  \min_{0\leq \ell < k \leq N}   \sum_{i=\ell}^{k-1} \log(1+K^TX(i)) \right]
				\end{align*} 
			which is desired.
			
			To prove part~$(ii)$, we note that $1+K^TX(i)$ is affine in~$K$ for any realization of $X(i)$ and $\log(\cdot)$ is concave. Hence, the composition $\log(1+K^TX(i))$ is concave in $K$. By the fact that the sum of concave functions is still concave and the pointwise minimum of concave functions is concave, it follows that the term   $\min_{0 \leq \ell < k \leq  N} \sum_{i=\ell}^{k-1} \log(1+K^TX(i))$ is concave in $K$. Finally, the expected value operator preserves concavity, hence, the surrogate expected drawdown $\mathbb{E}[\log D_K]$ is concave in $K$.
			This implies that the set
			$
			\mathcal{D}_{K, \delta} =\{K \in \mathbb{R}^n: \mathbb{E}[\log D_K] \geq \log(1-\delta)\}
			$ 
			is convex. 
\end{proof}

		\begin{theorem}[A Drawdown-Based Log-Optimal Portfolio Problem]
				Let $\delta \in (0,1)$ be given. Then the distributional robust log-optimal portfolio problem involving the surrogate drawdown constraint; i.e.,
				\begin{align*}
						\max_{K \in \mathcal{K} \cap \mathcal{D}_{K, \delta}} \; \inf_{p \in \mathcal{P}} g_p(K)
					\end{align*}
				is a concave optimization problem 	where $g_p(K) = \mathbb{E}_p[\log (1+K^T X(0))]$.
			\end{theorem}
		
\begin{proof}
Since $\mathbb{E}_p[\log (1+K^T X(0))]$ is concave in $K$, the infimum of a family of concave functions is concave.
Hence,	the objective function $\inf_{p \in \mathcal{P}} g_p(K)$ is  concave in $K$.
 By Lemma~\ref{lemma: convex drawdown surrogate}, the set $\mathcal{D}_{K, \delta}$ is convex. 
			Since~$\cal K$ is convex, $\mathcal{K} \cap \mathcal{D}_{K, \delta}$ is again convex. Therefore, the maximization problem stated in the theorem above has a concave objective with a convex constraint, which leads to a concave optimization problem. 
\end{proof}
		
		\smallskip
		\begin{remark}\rm
			In contrast to the work in~\cite{chekhlov2005drawdown} that uses arithmetic (uncompounded) returns for calculating the drawdown risks. Here, we propose a convex drawdown surrogate that facilitates the optimization. 
			There are also some alternative approaches for controlling the drawdown within the expected log-optimal portfolio framework; e.g., see \cite{maclean1992growth} and \cite{hsieh2017inefficiency}.
			\end{remark}

\bigskip
\bibliographystyle{ieeetr}
\bibliography{refs.bib}

\end{document}